\documentclass[12pt,reqno]{amsart}
\usepackage{hyperref}
\usepackage{amssymb,amsmath,amsfonts,latexsym}
\usepackage{bm}
\usepackage[us,12hr]{datetime}
\usepackage{mathtools}
\allowdisplaybreaks
\usepackage{enumerate}
\usepackage{fancyhdr}

\usepackage{array,graphics,color,ragged2e}
\numberwithin{equation}{section}
\newtheorem{dfn}{Definition}[section]
\newtheorem{thm}[dfn]{Theorem}
\newtheorem{lma}[dfn]{Lemma}

\newtheorem{rmrk}[dfn]{Remark}

\DeclarePairedDelimiterX{\norm}[1]{\lVert}{\rVert}{#1}
\DeclarePairedDelimiterX{\bnorm}[1]{\big\lVert}{\big\rVert}{#1}
\DeclarePairedDelimiterX{\Bnorm}[1]{\Big\lVert}{\Big\rVert}{#1}

\newcommand{\R}{\mathbb{R}}
\newcommand{\N}{\mathbb{N}}

\newcommand{\Nat}{\mathbb{N}}

\newcommand{\Tr}{\operatorname{Tr}}

\newcommand{\Mcal}{\mathcal{M}}
\newcommand{\Hcal}{\mathcal{H}}
\newcommand{\Bcal}{\mathcal{B}}

\newcommand{\M}{\mathcal{M}}

\def\Re{{\mathrm{Re}\,}}

\begin{document}
	\title[Approximation of the spectral action functional]
{Approximation of the spectral action\\ functional in the case of $\tau$-compact\\ resolvents}

\author[Chattopadhyay] {Arup Chattopadhyay}
\address{Department of Mathematics, Indian Institute of Technology Guwahati, Guwahati, 781039, India}
\email{arupchatt@iitg.ac.in, 2003arupchattopadhyay@gmail.com}

\author[Pradhan]{Chandan Pradhan}
\address{Department of Mathematics, Indian Institute of Technology Guwahati, Guwahati, 781039, India}
\email{chandan.math@iitg.ac.in, chandan.pradhan2108@gmail.com}

\author[Skripka] {Anna Skripka}
\address{Department of Mathematics and Statistics, University of New Mexico, 311 Terrace Street NE, Albuquerque, NM 87106, USA}
\email{skripka@math.unm.edu}

\subjclass{Primary 47A55}

\keywords{Spectral action functional, multiple operator integral}

\date{July 13, 2023}
	\begin{abstract}
	We establish estimates and representations for the remainders of Taylor approximations of the spectral action functional $V\mapsto\tau(f(H_0+V))$ on bounded self-adjoint perturbations, where $H_0$ is a self-adjoint operator with $\tau$-compact resolvent in a semifinite von Neumann algebra and $f$ belongs to a broad set of compactly supported functions including $n$-times differentiable functions with bounded $n$-th derivative. Our results significantly extend analogous results in \cite{SkAnJOT}, where $f$ was assumed to be compactly supported and $(n+1)$-times continuously differentiable. If, in addition, the resolvent of $H_0$ belongs to the noncommutative $L^n$-space, stronger estimates are derived and extended to noncompactly supported functions with suitable decay at infinity.
\end{abstract}

\maketitle
\section{Introduction}
Let $\Mcal$ be a semifinite von Neumann algebra acting on a separable Hilbert space $\Hcal$ equipped with a normal faithful semifinite trace $\tau$. Let $H_0$ be a closed densely defined self-adjoint operator affiliated with $\Mcal$ and assume that its resolvent is $\tau$-compact. Examples of such operators include differential operators on compact Riemannian manifolds (see, e.g., \cite[Chapter 3, Section B]{BE86} or \cite[Chapter 3, Section 6]{Kato}) and generalized Dirac operators of unital spectral triples (see, e.g., \cite{{NuSkJST}}).
For a sufficiently nice function $f$ and a self-adjoint element $V$ in $\Mcal$, which models a gauge potential, we consider a spectral action functional $V\mapsto\tau(f(H_0+V))$. The latter was introduced in \cite{CC97} to encompass different field actions in noncommutative geometry and recently received considerable attention in the literature (see, e.g., \cite{EI18}, for methods, examples, and open problems). Counterparts of the spectral action functional also arise in problems of mathematical physics on deterministic and random Dirac and Schr\"{o}dinger operators (see, e.g., \cite{S21,ST19}) and in the study of the spectral flow \cite{ACD}.

A perturbation approach to the spectral action functional in the setting of the algebra of bounded linear operators $\Bcal(\Hcal)$ was taken in \cite{vNvS21a}, where a noncommutative analog of the Taylor series expansion served as a starting point to understanding the structure of gauge fluctuations. Analytical aspects of Taylor approximations of the spectral action functional in the $\Bcal(\Hcal)$ setting were also studied in \cite{SA,SJ}, and the results of \cite{vNvS21a,SA,SJ} served as a starting point for a one-loop quantization of the spectral action obtained in \cite{oneloop}.
There are also important examples of spectral triples that are successfully studied in a more general semifinite von Neumann algebra framework (see, e.g., \cite{AM06}), and there are interesting models with functions $f$ not included or not conveniently described by the framework of \cite{SA,SJ} (see, e.g., \cite{CCS}).

Motivated by the aforementioned results and outlooks, this paper extends the perturbation theory of the spectral action functional to the setting of a semifinite von Neumann algebra, significantly relaxes smoothness assumptions imposed on admissible functions $f$ in \cite{SA,SkAnJOT} in the case of $\tau$-compact resolvents and, in addition, removes the compact support assumption on $f$ in the case of $\tau$-Schatten-von Neumann resolvents. Our description of admissible $f$ is given in terms of their decay properties, does not resort to the Laplace transform and even property involved in \cite{SJ}, and is similar in nature to the description in \cite{vNvS21a}.

Given a natural number $n\in\mathbb{N}$ and suitable $f$ and $H_0$, we denote the $n$-th order Taylor remainder of the spectral action functional $V\mapsto\tau\big(f(H_0+V)\big)$ by
\begin{align}\label{a5}
	\tau\left(\mathcal{R}_{H_0,f,n}(V)\right)=\tau\Big(f(H_0+V)-f(H_0)
	-\sum_{k=1}^{n-1}\,\frac{1}{k!}\frac{d^k}{ds^k}\,f(H_0+sV)\big\lvert_{s=0}\Big).
\end{align}
In Theorem \ref{mainthm} we establish that
\begin{align}\label{intro1}
	\left|\tau\big(\mathcal{R}_{H_0,f,n}(V)\big)\right|\leq D_{a,b,n,H_0,V} \|f^{(n)}\|_\infty
\end{align} for every self-adjoint $H_0$ with $\tau$-compact resolvent
and every $n$-times differentiable compactly supported in $(a,b)$ function $f$ with bounded $f^{(n)}$ and we derive an upper bound on the constant $D_{a,b,n,H_0,V}$ revealing explicit dependence on $H_0,V,n,a,b$.
We also establish the trace formula
\begin{align}
	\label{tf}
	\tau\big(\mathcal{R}_{H_0,f,n}(V)\big)
	=\int_\R f^{(n)}(\lambda)\eta_{n,H_0,V}(\lambda)d\lambda
\end{align}
for every $n$-times differentiable compactly supported in $(a,b)$ function $f$ such that $f^{(n)}$ exists almost everywhere and $f^{(n)}\in L^2(\R)$. The real-valued function $\eta_{n,H_0,V}\in L^1((a,b))$ is determined by \eqref{tf} uniquely up to a polynomial summand of degree at most $n-1$.

In Theorem~\ref{resolvent trace thm} we relax the differentiability assumption and remove the support restriction on functions $f$ satisfying \eqref{tf} under the stronger assumption on the operator $H_0$. Namely, in the case when $(H_0-iI)^{-1}$ belongs to the $\tau$-Schatten-von Neumann ideal associated with $(\Mcal,\tau)$ (see the definition \eqref{tausvn}), we establish \eqref{tf} for $(n-1)$-times continuously differentiable functions with suitable decay at infinity.
The locally integrable real-valued function $\eta_{n,H_0,V}$ is determined by \eqref{tf} uniquely up to a polynomial summand of degree at most $n-1$ and it satisfies the estimate
\begin{align*}
	|\eta_n(x)|\leq \text{ const}_n\,(2+\|V\|)\, \|V\|^{n-1}\,\|(H_0-iI)^{-1}\|_n^n\,(1+|x|)^n
\end{align*}
for every $x\in\R$, where $\|\cdot\|$ denotes the operator norm and $\|\cdot\|_n$ the noncommutative $L^n$-norm.

Our bound \eqref{intro1} extends the analogous bound of \cite{SkAnJOT}, where the additional assumption $f\in C_c^{n+1}((a,b))$ was imposed.
The formula \eqref{tf} was earlier established under the additional restriction $f\in C_c ^3((a,b))$ in the case $n=1$ in \cite[Theorem 2.5]{ACD} and under the restriction $f\in C_c^{n+1}((a,b))$ in the case $n\ge 2$ in \cite{SkAnJOT}. Other results in this direction were obtained in \cite[Theorem 18 and Corollary 19]{SJ} and \cite[Theorems 3.4 and 3.10]{SA} in the particular case when $\Mcal=\Bcal(\Hcal)$.
The result of Theorem~\ref{resolvent trace thm} relaxes the differentiability assumption made in \cite[Theorem 4.1]{NuSkJST} and extends the trace formula from the setting of $(\Bcal(\Hcal),\Tr)$ to the setting of a general semi-finite $(\Mcal,\tau)$.

The paper is organized as follows: preliminary results are discussed in Section~2; our main results are proved in Section~3 and Section~4 under the assumptions that $H_0$ has $\tau$-compact resolvent and that the resolvent of $H_0$ belongs to the noncommutative $L^n$-space, respectively.

\section{Preliminaries and notation}
We denote positive constants by letters $c,d,C,D$ with subscripts indicating dependence on their parameters. For instance, the symbol $c_\alpha$ denotes a constant depending only on the parameter $\alpha$.

\paragraph{\bf Function spaces.}	
Let $n\in\mathbb N$ and let $X$ be an interval in $\mathbb{R}$ possibly coinciding with $\mathbb{R}$. Let $B(X)$ denote the space of all bounded functions on $X$, $C(X)$ the space of all continuous functions on $X$, $C_0(\R)$ the space of continuous functions on $\R$ decaying to $0$ at infinity, $C^n(X)$ the space of $n$-times continuously differentiable functions on $X$. Let $C_b^n(X)$ denote the subset of $C^n(X)$ of such $f$ for which $f^{(n)}$ is bounded. Let $C_c^n(\mathbb{R})$ denote the subspace of $C^n(\mathbb{R})$ consisting of compactly supported functions.
We also use the notation $C^0(\mathbb{R})=C(\mathbb{R})$. Let $a,b\in\mathbb{R}$. Let $C_c^n((a,b))$ denote the subspace of $C_c^n(\R)$ consisting of the functions whose closed support is contained in $(a,b)$, let $D_c^n((a,b))$ denote the space of $n$-times differentiable functions in $C_c((a,b))$, and let $F_c^n((a,b))$ denote the subspace of $C_c^{n-1}((a,b))$ consisting of $f$ such that $f^{(n)}$ exists almost everywhere in $(a,b)$ and $f^{(n)}\in L^2((a,b))$. We note that for every $f\in F_c^n((a,b))$ the function $f^{(n-1)}$ is absolutely continuous. We write $f(x)=o(g(x))$ if for all $\epsilon>0$, we have $|f(x)|\leq\epsilon g(x)$ for all $x$ outside a compact set depending on $\epsilon$.
	
	Given $f\in L^1(\R)$, let $\hat{f}$ denote the Fourier transform of $f$.
	We will use the well-known property that every $f\in C_c^n(\R)$ satisfies $\widehat{f^{(n-1)}}\in L^1(\R)$.
	
	We will need the following elementary lemma.
	\begin{lma}\label{l3}
		Let $a,b\in\R,~a<b$, $k\in\mathbb{N}$, and $f\in D_c^k((a,b))$ be such that $f^{(k)}\in B([a,b])$. Then,
		\begin{align*}
			\|f^{(j)}\|_{\infty}\leq (b-a)^{k-j} \|f^{(k)}\|_{\infty}, \quad j=0,\dots,k.
		\end{align*}
	\end{lma}

	\smallskip
	
	\paragraph{\bf Operators with $\tau$-compact resolvent.}
	In the sequel we fix a semifinite von Neumann algebra $\Mcal$ endowed with a normal faithful semifinite trace $\tau$ and acting on a separable Hilbert space $\Hcal$. An example of such pair $(\Mcal,\tau)$ is the algebra of bounded linear operators $\Bcal(\Hcal)$ on $\Hcal$ equipped with the canonical trace $\Tr$. A projection $P\in\Mcal$ is called $\tau$-finite if $\tau(P)<\infty$. We note that a $\Tr$-finite projection has finite rank.
	
	Let $\mu_t(A)$ denote the $t^{\text\rm th}$ generalized $s$-number \cite[Definition 2.1]{FK} of a $\tau$-measurable \cite[Definition 1.2]{FK} operator $A$ affiliated with $\Mcal$. An operator $A\in\Mcal$ is said to be {\it $\tau$-compact} if and only if $\lim_{t\rightarrow\infty}\mu_t(A)=0$. If $(\Mcal,\tau)=(\Bcal(\Hcal),\Tr)$, then the concept of $\tau$-compactness coincides with the concept of compactness.
	
	When $H$ is a closed densely defined self-adjoint operator affiliated with $\Mcal$, we briefly write that {\it $H$ is a self-adjoint operator affiliated with $\Mcal$.}
	We say that a self-adjoint operator $H$ {\it has $\tau$-compact resolvent} if $H$ is affiliated with $\Mcal$ and $(H-zI)^{-1}$ is $\tau$-compact for some and, hence, for each resolvent point $z$ of $H$.

	The following useful property is a consequence of the second resolvent identity (see, e.g., \cite[Lemma 1.3]{ACD}).
	
	\begin{lma}
		\label{l1}
		Let $H_0$ be a self-adjoint operator with $\tau$-compact resolvent and let $V$ be a self-adjoint operator in $\Mcal$. Then, $H=H_0+V$ also has $\tau$-compact resolvent.
	\end{lma}
	
	Let $E_H(\cdot)$ denote the spectral measure of a self-adjoint operator $H$.
	The following result is a consequence of Lemma \ref{l1} and \cite[Lemma 1.4]{ACD}.
	
	\begin{lma}\label{l2}
		Let $H_0$ be a self-adjoint operator with $\tau$-compact resolvent and let $V$ be a self-adjoint operator in $\Mcal$.
		Then, for every compact subset $\Delta\subset\mathbb{R}$, the spectral projections $E_{H_0}(\Delta)$ and $E_{H_0+V}(\Delta)$ are $\tau$-finite.
	\end{lma}
	
	Let $f\in C^n(\R)$. Recall that the divided difference of order $n$ is an operation on the function $f$ of one
	(real) variable, and is defined recursively as follows:
	\begin{align*}
		&f^{[0]}(\lambda)=f(\lambda),\\
		&f^{[n]}(\lambda_0,\lambda_1,\ldots,\lambda_n)\\
		&=\begin{cases*}
			\frac{f^{[n-1]}(\lambda_0,\lambda_1,\ldots,\lambda_{n-2},\lambda_n)-f^{[n-1]}(\lambda_0,\lambda_1,\ldots,\lambda_{n-2},\lambda_{n-1})}{\lambda_n-\lambda_{n-1}} \quad \text{if}\quad \lambda_n\neq\lambda_{n-1},\\
			\frac{\partial}{\partial \lambda}f^{[n-1]}(\lambda_0,\lambda_1,\ldots,\lambda_{n-2},\lambda)\big|_{\lambda=\lambda_{n-1}}\quad \text{if}\quad \lambda_n=\lambda_{n-1}.
		\end{cases*}
	\end{align*}
	
	Let $L^p$, $1\leq p<\infty$, denote the noncommutative $L^p$-space associated with $(\mathcal{M},\tau)$, that is,
	$$L^p=\{A\;{\rm is\; affiliated\; with}\; \mathcal{M}:\;\|A\|_p:=(\tau(|A|^p))^{1/p}<\infty\}$$
	and let $\mathcal{L}^p$ denote the $\tau$-Schatten-von Neumann ideal associated with $(\Mcal,\tau)$, that is,
	\begin{align}
		\label{tausvn}
		\mathcal{L}^p=L^p\cap\M.
	\end{align}
	When $p=\infty$ we use the convention $\mathcal{L}^\infty=\Mcal$, $\|\cdot\|_\infty=\|\cdot\|$.
	\smallskip
	
	\paragraph{\bf Multilinear operator integrals.}
	Let $H$ be a self-adjoint operator affiliated with $\mathcal{M}$ and let $f\in C^{n}(\R)$ be such that $\widehat{f^{(n)}}\in L^1(\R)$. Let $p_k\in[1,\infty], 1\leq k\leq n$, and $E_{l,m}= E_H\big(\big[\frac{l}{m},\frac{l+1}{m}\big)\big)$ for $m\in \mathbb{N}$ and $l\in \mathbb{Z}$.
	Define a multilinear map on $\mathcal{L}^{p_1}\times\cdots\times\mathcal{L}^{p_n}$ by
	\begin{align}\label{a1}
		&T^{H,\ldots,H}_{f^{[n]}}(V_1,V_2,\ldots,V_n)\\
		\nonumber
		&=\lim_{m\to\infty}\lim_{N\to\infty}\sum_{|l_0|,|l_1|,\ldots,|l_n|\leq N}f^{[n]}\Big(\frac{l_0}{m},\frac{l_1}{m},\ldots,\frac{l_n}{m}\Big)E_{l_0,m}V_1E_{l_1,m}\\\nonumber
		&\hspace{2.5in}\times V_2E_{l_2,m}\cdots V_n E_{l_n,m},
	\end{align}
	where the limits are evaluated in the norm $\|\cdot\|_p$, $\frac{1}{p}=\frac{1}{p_1}+\frac{1}{p_2}+\cdots+\frac{1}{p_n}$. The existence of the limits in \eqref{a1} is justified in \cite[Lemma 3.5]{PSS}. We call $T^{H,\ldots,H}_{f^{[n]}}$ defined in \eqref{a1} a {\it multilinear operator integral with symbol $f^{[n]}$} and write $T_{f^{[n]}}$ when there is no ambiguity which element $H$ is used.
	
	Discussion of multiple operator integrals, including those with more general symbols, and their applications can be found in \cite{ST19}.
	It was shown in \cite[Lemma 3.1, Theorem 3.2]{SA} that the multilinear operator integral given by \eqref{a1} is bounded for all $f\in C^{n+1}_c(\R)$ when $H$ has compact resolvent.
	
	The following estimate is a consequence of \cite[Theorem 5.3 and Remark 5.4]{PSS} and \cite[Theorem 4.4.7]{ST19}.
	
	\begin{thm}\label{inv-thm}
		Let $k\in\mathbb{N}$ and let $\alpha,\alpha_1,\ldots,\alpha_k\in(1,\infty)$ satisfy $\tfrac{1}{\alpha_1}+\cdots+\tfrac{1}{\alpha_k}=\tfrac{1}{\alpha}$. Let $H$ and $\tilde{H}$ be two self-adjoint operators affiliated with $\mathcal{M}$. Assume that $V_\ell\in\mathcal{L}^{\alpha_\ell},\, 1\leq \ell\leq k$. Then there exists $c_{\alpha,k}>0$ such that
		\begin{align}\label{est}
			\|T^{\tilde{H},H,\ldots,H}_{f^{[k]}}(V_1,V_2,\ldots,V_k)\|_{\alpha}\leq c_{\alpha,k}\, \|f^{(k)}\|_\infty \prod\limits_{1\leq \ell \leq k}\|V_\ell\|_{\alpha_\ell}
		\end{align}
		for every $f\in C_b^k(\R).$
	\end{thm}
	
	We also have the following bound for the seminorm $|\tau(\cdot)|$ of a multilinear operator integral.
	
	\begin{thm}\label{est-thm}
		Let $k\in\mathbb{N}$ and let $\alpha_1,\ldots,\alpha_k\in(1,\infty)$ satisfy $\tfrac{1}{\alpha_1}+\cdots+\tfrac{1}{\alpha_k}=1$. Let $H$ be a self-adjoint operator affiliated with $\mathcal{M}$. Assume that $V_\ell\in\mathcal{L}^{\alpha_\ell},\, 1\leq \ell\leq k$. Then, for $c_{1,k}>0$ satisfying \eqref{est},
		\begin{align}\label{ss}
			|\tau(T^{H,H,\ldots,H}_{f^{[k]}}(V_1,\ldots,V_k))|&\leq c_{1,k}\|f^{(k)}\|_\infty\prod\limits_{1\leq \ell \leq k}\|V_\ell\|_{\alpha_\ell}
		\end{align}
		for every $f$ with $f^{(k)}\in C_0(\R)$ satisfying $\widehat{f^{(k)}}\in L^1(\R)$.
	\end{thm}
	
	\begin{proof}
		By the definition of the multiple operator integral \eqref{a1} and cyclicity of the trace,
		\begin{align}\label{sss}
			\tau\big(T^{H,H,\ldots,H}_{f^{[k]}}(V_1,\ldots,V_k)\big)
			=\tau\big(T^{H,H,\ldots,H}_{{\tilde{f}}^{[k]}}(V_1,\ldots,V_{k-1})V_k\big),
		\end{align}
		where ${\tilde{f}}^{[k]}(\lambda_0,\lambda_1,\ldots,\lambda_{k-1})=f^{[k]}(\lambda_0,\lambda_1,\ldots,\lambda_{k-1},\lambda_{k-1})$. Therefore, by\\ H\"{o}lder's inequality, Theorem \ref{inv-thm}, and \cite[Remark 5.4]{PSS} applied to \eqref{sss}, we obtain \eqref{ss}.
	\end{proof}
	
	Let $a,b\in \R$, $a<b$, and $\epsilon>0$. Let $a_\epsilon=a-\epsilon,\, b_\epsilon=b+\epsilon$. Define the function $\Phi_\epsilon$ on $\R$ by
	\begin{align}\label{indicator}
		\Phi_\epsilon(x)=(h_1(x)- h_2(x))^4,
	\end{align} where
	\begin{align*}
		&h_1(x)=\dfrac{\int_{a_\epsilon}^{x}\Phi(t-a_\epsilon)\Phi(a-t) \,dt}{\int_{a_\epsilon}^{a}\Phi(t-a_\epsilon)\Phi(a-t) \,dt},\quad	h_2(x)=\dfrac{\int_{b}^{x}\Phi(t-b)\Phi(b_\epsilon-t) \,dt}{\int_{b}^{b_\epsilon}\Phi(t-b)\Phi(b_\epsilon-t) \,dt},\\
		& \Phi(x)=\begin{cases}
			e^{-\frac{1}{x}} & \text{ if } x>0,\\
			0 & \text{ if } x\leq 0.
		\end{cases}
	\end{align*}
	Note that $\Phi_\epsilon\lvert_{(a,b)}=1,\,\Phi_\epsilon^{1/4}\in C_c^\infty((a_\epsilon,b_\epsilon))$ and $\|\Phi_\epsilon\|_\infty=1$. Moreover, if $H$ has $\tau$-compact resolvent, then by the spectral theorem, $\Phi_\epsilon(H)\in\mathcal{{L}}^1$ and $\|\Phi_\epsilon(H)\|_1\leq \tau(E_H(a_\epsilon,b_\epsilon))\|\Phi_\epsilon\|_\infty$ (see \cite[Proposition 2.3]{SkAnJOT}).
	
	\begin{thm}
		Let $H_0$ be a self-adjoint operator with $\tau$-compact resolvent, $k\in\mathbb{N}$, and $V_1,V_2,\ldots,V_k\in\Mcal$. Let $a,b\in\R$, $a<b$, and $\epsilon>0$. Then
		\begin{align}\label{a55}
			\big|\tau\big(T^{H_0,\ldots,H_0}_{f^{[k]}}(V_1,V_2,\ldots,V_k)\big)\big|\leq C_{a,b,k,\epsilon,H_0}~\|f^{(k)}\|_\infty\prod_{\ell=1}^k \|V_\ell\|
		\end{align} for every $f\in F_c^{k+1}((a,b))$, where
		\begin{align}
			\label{Clabel}
			C_{a,b,k,\epsilon,H_0}=\left((k+1)2^k+ c_{2,k} \right)(b-a+1)^k\,d_{k,\epsilon,H_0}\,\big(1+\tau(E_{H_0}(a,b))\big),
		\end{align}
		where $c_{2,k}$ satisfies \eqref{est} and
		\begin{align*}
			&d_{k,\epsilon,H_0}=\max\Big\{\|\Phi_\epsilon(H_0)\|_1,\,\max\limits_{1\leq \ell \leq k}\frac{1}{\ell!}\big\|\,\widehat{\Phi_\epsilon^{(\ell)}}\,\big\|_{1}~\Big\},
		\end{align*}
		where $\Phi_\epsilon$ is given by \eqref{indicator}.
	\end{thm}
	\begin{proof}
		The proof follows along the lines of the proof of \cite[Theorem 3.1]{SkAnJOT}.
	\end{proof}
	\begin{rmrk}
		Note that the upper bound for $	\big|\tau\big(T^{H_0,\ldots,H_0}_{f^{[k]}}(V_1,V_2,\ldots,V_k) \big)\big|$ in \eqref{a55} is stronger than the upper bound stated in \cite[Theorem 3.1]{SkAnJOT}.
	\end{rmrk}
	
	The following result is a consequence of \cite[Theorems 4.4.6 and 5.3.5]{ST19}.
	
	\begin{thm}\label{th1}
		Let $n\in\N$ and let $f\in C^n(\R)$ be such that $f^{(k)},\,\widehat{f^{(k)}}\in L^1(\R)$, $k=0,1,\dots,n$. Let $H$ be a self-adjoint operator affiliated with $\mathcal{M}$ and $V$ a self-adjoint operator in $\mathcal{M}$. Then, the G\^{a}teaux derivative $\frac{d^k}{dt^k}f(H+tV)|_{t=0}$ exists in the uniform operator topology and admits the multiple operator integral representation
		\begin{align}
			\label{a4}
			\frac{1}{k!}\frac{d^k}{ds^k}f(H+sV)\big|_{s=t}=T_{f^{[k]}}^{H+tV,\dots,H+tV}(V,V,\dots,V),\quad 1\leq k\leq n.
		\end{align}
		Moreover, if $V\in\mathcal{L}^{k}$, then the above $k$-th derivative is an element in $\mathcal{L}^{1}$.
	\end{thm}
	
	\section{ Trace formulas for operators with $\tau$-compact resolvents}
	
	In this section we establish our first main result.
	
	\begin{thm}\label{mainthm}
		Let $H_0$ be a self-adjoint operator with $\tau$-compact resolvent, $V$ a self-adjoint operator in $\Mcal$, $n\in \mathbb{N}$, and let $-\infty<a<b<\infty$, $\epsilon>0$. Then,
		\begin{align}\label{a6}
			\left|\tau\big(\mathcal{R}_{H_0,f,n}(V)\big)\right|\leq D_{a,b,n,\epsilon,H_0,V} \|f^{(n)}\|_\infty
		\end{align}
		for every $f\in D_c^n((a,b))$ with $f^{(n)}\in B([a,b])$, where
		\begin{align}\label{const}
			&D_{a,b,n,\epsilon,H_0,V}\\
			\nonumber
			&=(b-a)^{n}\max\big\{\tau(E_{H_0}([a,b])),\tau(E_{H_0+V}([a,b]))\big\}\\
			\nonumber&\hspace{2in}+\sum_{k=1}^{n-1} (b-a)^{n-k}C_{a,b,k,\epsilon,H_0}\|V\|^k,
		\end{align}
		and $C_{a,b,k,\epsilon,H_0}$ satisfies \eqref{a55}. Furthermore, there exists a real-valued function $\eta_{a,b,n,\epsilon,H_0,V}\in L^1((a,b))$ such that
		\begin{align}
			\label{aa7}
			\tau\big(\mathcal{R}_{H_0,f,n}(V)\big)=\int_{a}^{b}f^{(n)}(\lambda)\eta_{a,b,n,\epsilon,H_0,V}(\lambda)d\lambda
		\end{align}
		for every $f\in F_c^n((a,b))$ and
		\begin{align}
			\label{etabound}
			\int_a^b\big|\eta_{a,b,n,\epsilon,H_0,V}(\lambda)\big|d\lambda\leq D_{a,b,n,\epsilon,H_0,V}.
		\end{align}
		The function $\eta_{a,b,n,\epsilon,H_0,V}\in L^1((a,b))$ is determined by \eqref{aa7} uniquely up to a polynomial summand of degree at most $n-1$.
	\end{thm}
	
	\begin{proof}
		By Lemma \ref{l2}, $E_{H_0}([a,b])$ and $E_{H_0+V}([a,b])$ are $\tau$-finite projections. Let $f\in F_c^n((a,b))$.
		
		Case 1: $n=1$. \\
		By the spectral theorem for a self-adjoint operator $H$ in $\mathcal{H}$ with $\tau$-compact resolvent, we have
		$$f(H)=f(H)E_{H}([a,b])=\int_{a}^{b}f(\lambda)\,dE_{H}(\lambda),$$
		where the integral converges in $\mathcal{L}^1$. Hence, by continuity of the trace $\tau$,
		\begin{align}\label{z}
			\tau\big(\mathcal{R}_{H_0,f,1}(V)\big)=\int_a^b f(\lambda)d\big(\tau(E_{H_0+V}(\lambda))-\tau(E_{H_0}(\lambda))\big).
		\end{align}
		Integrating by parts on the right-hand side of \eqref{z} and applying the support property
		and absolute continuity of $f$ yield
		\begin{align}
			\label{rem1}
			\nonumber
			\tau\big(\mathcal{R}_{H_0,f,1}(V)\big)
			&=-\int_a^b \big(\tau(E_{H_0+V}([a,\lambda)))-\tau(E_{H_0}([a,\lambda)))\big) df(\lambda)\\
			&=\int_a^b f'(\lambda)\big(\tau(E_{H_0}([a,\lambda)))-\tau(E_{H_0+V}([a,\lambda)))\big)d\lambda.
		\end{align}
		Thus, \eqref{aa7} holds for $n=1$ with
		\begin{align}
			\label{etavianu1}
			\eta_{a,b,1,\epsilon,H_0,V}(\lambda)=\tau(E_{H_0}([a,\lambda)))-\tau(E_{H_0+V}([a,\lambda)))
		\end{align}
		and \eqref{etabound} holds with
		\begin{align*}
			D_{a,b,1,\epsilon,H_0,V}=(b-a)\max\big\{\tau(E_{H_0}([a,b])),\tau(E_{H_0+V}([a,b]))\big\}.
		\end{align*}
		
		Case 2: $n\ge 2$. \\
		By \eqref{rem1},
		\begin{align*}
			&\big|\tau(f(H_0+V))-\tau(f(H_0))\big|\\
			\le\,&\|f'\|_\infty(b-a)
			\max\big\{\tau(E_{H_0}([a,b])),\tau(E_{H_0+V}([a,b]))\big\}.
		\end{align*}
		The latter along with Lemma \ref{l3} implies
		\begin{align}
			\label{33}
			&\big|\tau(f(H_0+V))-\tau(f(H_0))\big|\\
			\nonumber
			&\le\|f^{(n-1)}\|_\infty(b-a)^{n-1}\max\big\{\tau(E_{H_0}([a,b])),\tau(E_{H_0+V}([a,b]))\big\}.
		\end{align}
		Combining \eqref{a55} and \eqref{a4} and then applying Lemma \ref{l3} yield
		
		\begin{align}
			\label{34}
			\nonumber
			\left|\tau\left(\frac{1}{k!}\frac{d^k}{ds^k}f(H_0+sV)\big|_{s=0}\right)\right|&\leq ~C_{a,b,k,\epsilon,H_0}~\|f^{(k)}\|_\infty\|V\|^k\\
			&\leq(b-a)^{n-k-1}C_{a,b,k,\epsilon,H_0}\|f^{(n-1)}\|_\infty\|V\|^k
		\end{align}
		for every $k=1,\dots,n-1$, where $C_{a,b,k,\epsilon,H_0}$ satisfies \eqref{Clabel}.
		Combining \eqref{a5}, \eqref{33}, and \eqref{34} implies
		\begin{align}\label{a9}
			\left|\tau\big(\mathcal{R}_{H_0,f,n}(V)\big)\right|\le\widetilde{D}_{a,b,n,\epsilon,H_0,V}\|f^{(n-1)}\|_\infty,
		\end{align}
		where
		\begin{align*}
			&\widetilde{D}_{a,b,n,\epsilon,H_0,V}\\
			&\quad=(b-a)^{n-1}\max\big\{\tau(E_{H_0}([a,b])),\tau(E_{H_0+V}([a,b]))\big\}\\
			&\hspace{2in}+\sum_{k=1}^{n-1} (b-a)^{n-k-1}C_{a,b,k,\epsilon,H_0}\|V\|^k.
		\end{align*}
		
		By the Riesz representation theorem for elements in $\big(C_0(\mathbb{R})\big)^*$, Hahn-Banach theorem, and estimate \eqref{a9}, there exists a finite (complex) measure $\nu_{a,b,n,\epsilon,H_0,V}$ such that
		\begin{align}\label{a8}
			\tau\big(\mathcal{R}_{H_0,f,n}(V)\big)=\int_{a}^{b}f^{(n-1)}(\lambda)d\nu_{a,b,n,\epsilon,H_0,V}(\lambda)
		\end{align}
		for every $f\in F_c^n((a,b))$ and the total variation of $\nu_{a,b,n,\epsilon,H_0,V}$ is bounded by
		\begin{align}
			\label{nubound}
			\|\nu_{a,b,n,\epsilon,H_0,V}\|\leq \widetilde{D}_{a,b,n,\epsilon,H_0,V}.
		\end{align}
		Integrating by parts on the right-hand side of \eqref{a8} and applying the support property of $f$
		and absolute continuity of $f^{(n-1)}$ yield
		\begin{align}
			\label{a10}
			\tau\big(\mathcal{R}_{H_0,f,n}(V)\big)
			=\int_{a}^{b}f^{(n)}(\lambda)\big(-\nu_{a,b,n,\epsilon,H_0,V}([a,\lambda))\big)d\lambda.
		\end{align}
		Thus, \eqref{a10} implies \eqref{aa7} with
		\begin{align}
			\label{etavianu}
			\eta_{a,b,n,\epsilon,H_0,V}(\lambda)=-\nu_{a,b,n,\epsilon,H_0,V}([a,\lambda)).
		\end{align}
		Combining \eqref{nubound} and \eqref{etavianu} ensures \eqref{etabound}, where
		\begin{align}
			\label{D=tD}
			D_{a,b,n,\epsilon,H_0,V}=(b-a)\widetilde{D}_{a,b,n,\epsilon,H_0,V}.
		\end{align}
		
		Let $\tilde{\eta}_{a,b,n,\epsilon,H_0,V}:=\Re( \eta_{a,b,n,\epsilon,H_0,V})$. Since the left-hand side of \eqref{aa7} is real-valued whenever $f$ is real-valued, we obtain that
		$\tilde{\eta}_{a,b,n,\epsilon,H_0,V}$ satisfies \eqref{aa7} and \eqref{etabound} for real-valued $f\in F_c^n((a,b))$ and, consequently, for all $f\in F_c^n((a,b))$. Therefore, without loss of generality we may consider $\eta_{a,b,n,\epsilon,H_0,V}$ to be real-valued satisfying \eqref{aa7} and \eqref{etabound}. Next,
		suppose there exists another real-valued function $\xi_{a,b,n,H_0,V}\in L^1((a,b))$ satisfying \eqref{aa7}. Let $h_n=\eta_{a,b,n,\epsilon,H_0,V}-\xi_{a,b,n,H_0,V}$. Then, it follows from \eqref{aa7} that
		\begin{align}\label{a11}
			\int_{a}^{b}f^{(n)}(\lambda)h_n(\lambda)d\lambda=0\quad \text{for all }f\in C_c^\infty((a,b)).
		\end{align}
		Consider the distribution $T_{h_n}$ defined by 	
		$$T_{h_n}(\phi)=\int_{a}^{b} \phi\, h_n\,d\lambda$$ for every $\phi\in C^\infty_c((a,b))$.
		By \eqref{a11} and the definition of the derivative of a distribution, $T_{h_n}^{(n)}=0$. Hence by \cite[Theorem 3.10 and Example 2.21]{gwaiz}, $h_n$ is a polynomial of degree at most $n-1$. Consequently, $\eta_{a,b,n,\epsilon,H_0,V}\in L^1((a,b))$ satisfying \eqref{aa7} is unique up to an additive polynomial of degree at most $n-1$.
		
		Assume now that $n\in\mathbb{N}$ and $f\in D_c^n((a,b))$ with $f^{(n)}\in B([a,b])$.
		Applying Lemma \ref{l3} in \eqref{a9} yields \eqref{a6}, completing the proof of the theorem.
	\end{proof}
	\begin{rmrk}
		It follows from the proof of Theorem \ref{mainthm} that the representation \eqref{aa7} with $n=1$ holds for a larger class of functions $f$, namely, for every $f$ absolutely continuous on $[a,b]$ and compactly supported in $(a,b)$.
	\end{rmrk}
	\begin{rmrk}
		The uniqueness of the spectral shift function $\eta_{a,b,n,\epsilon,H_0,V}$ up to an additive polynomial of degree at most $n-1$ was not addressed in \cite[Theorem 4.3]{SkAnJOT}.
	\end{rmrk}
	
	\section{The case of resolvents in the $\tau$-Schatten-von Neumann ideal}
	
	In this section, we obtain an upper bound for $|\tau(\mathcal{R}_{H_0,f,n}(V))|$ that is independent of the support of $f$ under the additional assumption that the resolvent of $H_0$ belongs to $\mathcal{L}^n$, $n\in \mathbb{N}$. As an application, we extend the trace formula \eqref{aa7} to a larger class of scalar functions $f$ and obtain a pointwise bound on the spectral shift function.
	
	The trace formula \eqref{tf} was obtained in \cite{NuSkJST} for the class of scalar functions $\mathfrak{W}_n$ defined below under the assumption that $V(H_0-iI)^{-1}$ belongs to the Schatten-von Neumann ideal associated with $(\Bcal(\Hcal),\Tr)$.
	\begin{dfn}
		\label{fracw}
		Let $n\in\mathbb{N}$. Let $\mathfrak{W}_n$ denote the set of functions $f\in C^n(\R)$ such that
		\begin{enumerate}[(i)]
			\item\label{fracwi} $\widehat{f^{(k)}u^{k}}\in L^{1}(\R)$, $k=0,1,\ldots, n$,
			
			\item\label{fracwii} $f^{(k)}\in L^1\big(\R,(1+|x|)^{k-1}\,dx\big)$, $k=1,\ldots,n$.
		\end{enumerate}
	\end{dfn}
	
	As noted in Remark \ref{rk} below, $F_c^{n}(\R)\not\subset \mathfrak{W}_n$. Our principle goal in this section is to establish the trace formula \eqref{a5} for a larger set of functions containing $F_c^{n}(\R)$. In this context, we introduce the following class of functions.
	
	\begin{dfn}
		Let $n\in \mathbb{N}$. Let $\mathfrak{H}_n$ denote the set of functions $f\in C^{n-1}(\R)$ such that
		\begin{enumerate}[(i)]
			\item $\widehat{f^{(k)}u^{k}}, \widehat{f^{(k)}u^{k+1}}\in L^{1}(\R)$, $k=0,1,\ldots, n-1$,
			\item $f^{(n)}$ exists almost everywhere,
			\item $f^{(k)}\in L^1\big(\R,(1+|x|)^{k}\,dx\big)$, $k=0,1,\ldots,n$.
		\end{enumerate}
	\end{dfn}
	
	
	\begin{rmrk}\label{rk}
		We have
		\begin{align}
			\label{inclusion}
			F_c^{n}(\R)\subset\mathfrak{H}_n\subset \mathfrak{W}_{n-1},
		\end{align}
		but
		\begin{align}
			\label{notinclusion}
			F_c^{n}(\R)\not\subset \mathfrak{W}_{n},\quad \mathfrak{H}_n\not\subset\mathfrak{W}_n,\quad \mathfrak{W}_n\not\subset\mathfrak{H}_n.
		\end{align}
		The inclusions \eqref{inclusion} follow directly from the definitions of the sets.
		Note that the function
		\[f(x)=\begin{cases}
			0 &\text{ if } x<0,\\
			x^n(x-1)^n &\text{ if } 0\leq x\leq 1,\\
			0 &\text{ if } x>1,
		\end{cases}\]
		satisfies $f\in F_c^n(\R)$ but $f\not\in C^n(\R)$. Hence, $f\not\in\mathfrak{W}_n$, so the first two properties in \eqref{notinclusion} hold. Note that $g(x)=(x-i)^{-1}\in \mathfrak{W}_{n}$ but $g\notin \mathfrak{H}_{n}$. Therefore, the third property in \eqref{notinclusion} holds.
	\end{rmrk}
	
	Below we will also use the notations $u(\lambda):=(\lambda-i)$ and $u^{k}(\lambda)=(u(\lambda))^k,\, k\in\mathbb{Z},\, \lambda\in\R.$
	
	\begin{lma}\label{convl}
		Let $f\in \mathfrak{H}_n$. Then, $\widehat{f^{(k)}u^l}\in L^1({\R})$ for $0\leq l\leq k\leq n-1$.
	\end{lma}
	\begin{proof}
		For $k=l$, the result follows from the definition of $\mathfrak{H}_n$.
		
		Let $l<k$. Since $u^{-k},u^{-k-1}\in L^2(\R)$, we obtain $\widehat{u^{-k}}\in L^1(\R)$ for all $k\in\N$ (see, e.g., \cite[Lemma 7]{PoSu09}).
		By the convolution theorem,
		\[\widehat{f^{(k)}u^l}=\widehat{\big(f^{(k)}u^{k}u^{l-k}\big)}=\widehat{f^{(k)}u^{k}}*\widehat{u^{l-k}}.\]
		Since $L^1(\R)$ is closed under the convolution product,
		$\widehat{f^{(k)}u^l}\in L^1(\R)$.
	\end{proof}
	
	\begin{lma}\label{exp-lem}
		Let $n\in\mathbb{N}$. Let $H_0$ be a self-adjoint operator affiliated with $\mathcal{M}$ and let $V$ be a self-adjoint operator in $\mathcal{M}$. Let $H_t=H_0+tV,\, t\in\R$, and $\tilde{V}=V(H_0-iI)^{-1}$. Then
		\begin{align}\label{ddd1}
			\nonumber&T^{H_t,H_0,\ldots,H_0}_{f^{[n-1]}}(V,\ldots,V)
			=(-1)^{n-1}f(H_t)\,\tilde{V}^{n-1}\\
			\nonumber&+\sum_{p=1}^{n-1}\sum_{\substack{j_1,\ldots,j_{p}\geq1,j_{p+1}\geq0\\j_1+\cdots+j_{p+1}=n-1}}(-1)^{n-p-1}\, \Big(T^{H_t,H_0,\ldots,H_0}_{(fu^{p+1})^{[p]}}(\tilde{V}^{j_1},\ldots,
			\tilde{V}^{j_{p}}(H_0-iI)^{-1})\\
			&\hspace{2cm}\times\tilde{V}^{j_{p+1}}-T^{H_t,H_0,\ldots,H_0}_{(fu^{p})^{[p-1]}}(\tilde{V}^{j_1},\ldots,\tilde{V}^{ j_{p-1}})\tilde{V}^{j_p}(H_0-iI)^{-1}\tilde{V}^{j_{p+1}}\Big)
		\end{align}
		for all $f\in\mathfrak{H}_n$.
		
	\end{lma}
	
	\begin{proof}
		The identity \eqref{ddd1} is trivial in the case $n=1$.
		
		Let $n\geq 2$. It was noted in \cite[Equation (24)]{NuSkJST} that
		\begin{align}	
			\label{dd1}	
			&f^{[n-1]}(\lambda_0,\ldots,\lambda_{n-1})\\
			\nonumber
			&=\sum_{p=0}^{n-1}\sum_{0<j_1<\cdots<j_{p}\leq n-1}(-1)^{n-1-p}
			(fu^p)^{[p]}(\lambda_0,\lambda_{j_1},\ldots,\lambda_{j_p})\\
			\nonumber &\hspace{2.5in}\times u^{-1}(\lambda_1)\cdots u^{-1}(\lambda_{n-1}).
		\end{align}
		By Lemma \ref{convl}, $\widehat{f^{(n-1)}},\,\widehat{(fu^p)^{(p)}}\in L^1(\R)$ for $0\leq p\leq n-1$. Therefore, applying \cite[Lemmas 3.5, 5.1, 5.2]{PSS} yields
		\begin{align}\label{dd2}
			&T^{H_t,H_0,\ldots,H_0}_{f^{[n-1]}}(V,\ldots,V)\\
			\nonumber
			&=(-1)^{n-1}f(H_t)\tilde{V}^{n-1}\\
			\nonumber&\qquad+\sum_{p=1}^{n-1}\sum_{\substack{j_1,\ldots,j_{p}\geq1,j_{p+1}\geq0\\j_1+\cdots+j_{p+1}=n-1}}\!(-1)^{n-1-p}\, T^{H_t,H_0,H_0,\ldots,H_{0}}_{(fu^p)^{[p]}}(\tilde{V}^{j_1},\ldots,\tilde{V}^{j_p})\,	\tilde{V}^{j_{p+1}}.
		\end{align}
		By Lemma \ref{convl} $\widehat{(fu^p)^{(p)}}, \widehat{(fu^{p+1})^{(p)}},$ and $\widehat{(fu^{p})^{(p-1)}}\in L^1(\R)$ for $1\leq p\leq n-1$. Therefore, applying \cite[Theorem 3.10(i)]{NuSkJST} to \eqref{dd2} yields \eqref{ddd1}.
	\end{proof}
	
	\begin{thm}\label{resolvent trace thm}
		Let $n\in\N$, let $H_0$ be a self-adjoint operator affiliated with $\mathcal{M}$ such that $(H_0-iI)^{-1}\in\mathcal{L}^{n}$, and let $V$ be a self-adjoint operator in $\mathcal{M}$. Then, there exists constant $K_n>0$ and a real-valued function $\eta_n$ such that
		\begin{align}\label{rr0}
			|\eta_n(x)|\leq K_n\,(2+\|V\|)\, \|V\|^{n-1}\,\|(H_0-iI)^{-1}\|_n^n\,(1+|x|)^n,\quad x\in\R,
		\end{align} and
		\begin{align}
			\label{rr1}
			\tau(\mathcal{R}_{H_0,f,n}(V))=\int_\R f^{(n)}(x)\eta_n(x)\,dx\,
		\end{align}
		for every $f\in \mathfrak{H}_{n}\cup \mathfrak{W}_n$.
		The locally integrable function $\eta_n$ is determined by \eqref{rr1} uniquely up to a polynomial summand of degree at most $n-1$.
	\end{thm}
	
	\begin{proof}
		The result for $\mathcal{M}=\mathcal{B}(\mathcal{H})$ and $f\in\mathfrak{W}_n$ is established in \cite[Theorem 4.1]{NuSkJST}. It extends to the case of a general semifinite $(\mathcal{M},\tau)$ by replacing results for $\mathcal{B}(\mathcal{H})$ with completely analogous results for $\mathcal{M}$ (see Theorems \ref{inv-thm}, \ref{est-thm}, \ref{th1} and \cite[Lemma 5.1]{PaWiSu02}).
		Therefore, there exists a real-valued function $\eta_{_{\mathfrak{W}_n}}$, unique up to a polynomial summand of degree at most $n-1$, satisfying \eqref{rr1} for all $f\in\mathfrak{W}_n$.
		
		\noindent Now we assume that $f\in\mathfrak{H}_n$. If $n=1$, then
		\begin{align}\label{dddd1}
			&\nonumber\lvert\tau(\mathcal{R}_{H_0,f,n}(V))\rvert\\
			\nonumber =&\,\lvert\tau(f(H_0+V)-f(H_0))\rvert\\
			\nonumber=&\,\lvert\tau((fu)(H_0+V)(H_0+V-iI)^{-1}-(fu)(H_0)(H_0-iI)^{-1})\rvert\\
			\leq&\, \|fu\|_\infty\, (2+\|V\|)\,\|(H_0-iI)^{-1}\|_1.
		\end{align}
		Let $n\geq 2$. By Theorem \ref{th1} and \cite[Theorem 4.3.14]{ST19},
		\begin{align}\label{r1}
			\nonumber
			\mathcal{R}_{H_0,f,n}(V)=&\mathcal{R}_{H_0,f,n-1}(V)-\frac{1}{(n-1)!}\frac{d^{n-1}}{ds^{n-1}}f(H_0+sV)\big|_{s=0}\\
			=&T_{f^{[n-1]}}^{H_0+V,H_0,\dots,H_0}(V,\ldots,V)-T_{f^{[n-1]}}^{H_0,\dots,H_0}(V,\ldots,V).
		\end{align}
		Let $H_t=H_0+tV,\, t\in\R$, and $\tilde{V}=V(H_0-iI)^{-1}$.
		By Lemma \ref{exp-lem} applied to \eqref{r1},
		\begin{align}\label{r3}
			&\mathcal{R}_{H_0,f,n}(V)\\		
			\nonumber	&=(-1)^{n-1}T_{f^{[1]}}^{H_0+V,H_0}(V)\tilde{V}^{n-1}+\sum_{p=1}^{n-1}\sum_{\substack{j_1,\ldots,j_{p}\geq1,j_{p+1}\geq0\\j_1+ \cdots+j_{p+1}=n-1}}(-1)^{n-p-1}\\
			\nonumber&\times\Big[\Big(T^{H_0+V,H_0,\ldots,H_0}_{(fu^{p+1})^{[p]}}(\tilde{V}^{j_1},\ldots,\tilde{V}^{j_{p}}(H_0-iI)^{-1})\\
			\nonumber&\hspace{1.5in}-T^{H_0,H_0,\ldots,H_0}_{(fu^{p+1})^{[p]}}(\tilde{V}^{j_1},\ldots,\tilde{V}^{j_{p}}(H_0-iI)^{-1})\Big)\tilde{V}^{j_{p+1}}\\
			\nonumber			\nonumber&\quad-\Big(T^{H_0+V,H_0,\ldots,H_0}_{(fu^{p})^{[p-1]}}(\tilde{V}^{j_1},\ldots,\tilde{V}^{j_{ p-1}})-T^{H_0,H_0,\ldots,H_0}_{(fu^{p})^{[p-1]}}(\tilde{V}^{j_1},\ldots,\tilde{V}^{ j_{p-1}})\Big)\\
			\nonumber&\hspace{3in}\times{\tilde{V}^{j_p}(H_0-iI)^{-1}}\tilde{V}^{j_{p+1}}\Big].
		\end{align}
		Note that, if $j_{p+1}=0$, then by definition of the multiple operator integral we have
		\begin{align}\label{new}
			\nonumber&T^{H_0+V,H_0,\ldots,H_0}_{(fu^{p+1})^{[p]}}(\tilde{V}^{j_1},\ldots,\tilde{V}^{j_{p}}(H_0-iI)^{-1})\\
			&=T^{H_0+V,H_0,\ldots,H_0}_{(fu^{p+1})^{[p]}}(\tilde{V}^{j_1},\ldots,\tilde{V}^{j_{p}})(H_0-iI)^{-1}.
		\end{align}
		By Lemma \ref{convl}, $\widehat{f},\,\widehat{f'},\,\widehat{(fu)'}\in L^1(\R)$. Applying \cite[Theorem 3.10(i)]{NuSkJST} yields
		\begin{align}\label{dd}
			T_{f^{[1]}}^{H_0+V,H_0}(V)=T_{(fu)^{[1]}}^{H_0+V,H_0}(\tilde{V})-f(H_0+V)\tilde{V}.
		\end{align}
		Combining \eqref{r3}, \eqref{new}, \eqref{dd}, Theorems \ref{inv-thm} and \ref{est-thm}, and applying H\"{o}lder's inequality yields
		\begin{align}\label{dddd2}
			\nonumber&\left|\tau(\mathcal{R}_{H_0,f,n}(V))\right|\\
			\nonumber&\leq\Big[\left(\|f\|_\infty+\|(fu)'\|_\infty\right)\|V\|^{n}\\
			\nonumber&\hspace{.5in}+\sum_{p=1}^{n-1}d_{p,n}\left(\|(fu^{p+1})^{(p)}\|_\infty
			+\|(fu^{p})^{(p)}\|_\infty\right)\|V\|^{n-1}\Big]\|(H_0-iI)^{-1}\|_n^n\\
			\nonumber&\leq C_n\Big(\sum_{p=1}^{n-1}\|(fu^{p+1})^{(p)}\|_\infty
			+\sum_{p=0}^{n-1}\|(fu^{p})^{(p)}\|_\infty\Big)\\
			&\hspace{2in}\times(1+\|V\|)\|V\|^{n-1}\|(H_0-iI)^{-1}\|_n^n,
		\end{align}
		where $C_n$ is some constant depending only on $n$.
		Arguing similarly to the proof of the existence of the spectral shift function in \cite[Proposition 4.2]{NuSkJST} we obtain from \eqref{dddd1} and \eqref{dddd2} that for each $n\in\Nat$,
		\begin{align}\label{r6}
			\tau\left(\mathcal{R}_{H_0,f,n}(V)\right)=\int_{\R} f^{(n)}(x)\acute\eta_n(x)dx,
		\end{align}
		where $\acute\eta_n$ is a continuous function on $\R$ such that
		\begin{align}\label{r7}
			|\acute\eta_n(x)|\leq D_n\,(2+\|V\|)\|V\|^{n-1}\|(H_0-iI)^{-1}\|_n^n(1+|x|)^n,\quad x\in\R{,}
		\end{align}
		for some constant $D_n>0$. We define $$\eta_{_{\mathfrak{H}_n}}:=\Re(\acute\eta_n).$$ Then it is clear that $\eta_{_{\mathfrak{H}_n}}$ satisfies \eqref{r7} as $|\eta_{_{\mathfrak{H}_n}}|\leq|\acute\eta_{n}|$.
		
		Since the left-hand side of \eqref{r6} is real-valued whenever $f$ is real-valued, we obtain that
		$\eta_{_{\mathfrak{H}_n}}$ satisfies \eqref{rr1} for real-valued $f\in \mathfrak{H}_{n}$ and, consequently, for all $f\in \mathfrak{H}_{n}$.
		The uniqueness of $\eta_{_{\mathfrak{H}_n}}$ satisfying \eqref{rr1} up to a polynomial summand of degree at most $n-1$ can be established completely analogously to the uniqueness of the function $\eta_{a,b,n,\epsilon,H_0,V}$ established in Theorem \ref{mainthm}. Since both $\eta_{_{\mathfrak{H}_n}}$ and $\eta_{_{\mathfrak{W}_n}}$ satisfy \eqref{rr1} for all $f\in C_c^{\infty}(\R)$, by using properties of distributions as in the proof of Theorem \ref{mainthm}, we conclude that
		$\eta_{_{\mathfrak{H}_n}}-\eta_{_{\mathfrak{W}_n}}=Q_n$, where $Q_n$ is a polynomial of degree at most $n-1$.
		By Definition \ref{fracw} and integration by parts,
		$\int_{\R} f^{(n)}(x)Q_n(x)dx=0$ for each $f\in \mathfrak{W}_n$. Therefore,
		\begin{align*}
			\int_{\R} f^{(n)}(x)\eta_{_{\mathfrak{W}_n}}dx=\int_{\R} f^{(n)}(x)(\eta_{_{\mathfrak{H}_n}}-Q_n)dx
			=\int_{\R} f^{(n)}(x)\eta_{_{\mathfrak{H}_n}}dx,
		\end{align*}
		for each $f\in \mathfrak{W}_n$. Hence $\eta_n:=\eta_{_{\mathfrak{H}_n}}$ satisfies \eqref{rr0} and \eqref{rr1} for all $f\in \mathfrak{H}_n\cup\mathfrak{W}_n$.
	\end{proof}
	
	\section*{Declarations} 	
	
	{\textbf{Conflicts of interests:}}	The authors declare that they have no conflict of interest. No datasets were generated or analyzed during the current study.
	
	
	\subsection*{Acknowledgment}
	\begin{justify}
		{\it The research of the first named author is supported by the Mathematical Research Impact Centric Support (MATRICS) grant, File No: MTR/2019/\\000640, by the Science and Engineering Research Board (SERB), Department of Science $\&$ Technology (DST), Government of India. The second named author gratefully acknowledge the support provided by IIT Guwahati, Government of India. The research of the third named author is supported in part by NSF grant DMS-1554456.}
	\end{justify}

\end{document}